\documentclass[reqno]{amsart}
\usepackage{amsthm,stmaryrd}
\usepackage{amssymb}
\usepackage{epsfig}
\usepackage[usenames,dvipsnames]{color}
\usepackage{verbatim}
\usepackage{mathrsfs}
\usepackage[all]{xy}
\usepackage{xcolor}
\newcommand{\coh}{\omega}
\newcommand{\Lag}{{\mathscr L}}
\newcommand{\ds}[1]{{\pmb{#1}}}
\newcommand{\Zr}{{\mathsf Z}}
\newcommand{\dep}{\delta}
\newcommand{\vp}{\varphi}
\newcommand{\Cob}{\mathcal{C}ob}
\newcommand{\VV}{\mathbb{V}}
\newcommand{\Pa}{\mathsf{P}}

\newcommand{\Hr}{H_r}

\newcommand{\hS}{\wh{S}}
\newcommand{\wh}{\widehat}
\newcommand{\wb}{\overline}

\newcommand{\disc}{\mathrm{disc}}

\newcommand{\E}{\mathscr{E}}
\newcommand{\lk}{\operatorname{lk}}
\newcommand{\brk}[1]{{\left\langle{#1}\right\rangle}}

\newcommand{\ve}{\varepsilon}
\newcommand{\Nr}{{\mathsf N}}
\newcommand{\eexp}{{\operatorname{e}}}
\newcommand{\slt}{{\mathfrak{sl}(2)}}
\newcommand{\Ubar}{{\wb U_q^{H}\slt}}
\newcommand{\cat}{\mathscr{C}}
\newcommand{\gcat}{\mathfrak{C}}
\newcommand{\Id}{\operatorname{Id}}
\newcommand{\bp}[1]{{\left(#1\right)}}
\newcommand{\qn}[1]{{\left\{#1\right\}}}
\newcommand{\de}{{k}}
\newcommand{\qd}{{\mathsf d}}
\newcommand{\Hom}{\operatorname{Hom}}
\newcommand{\HHom}{\operatorname{\mathbb{H}om}}
\newcommand{\Hoch}{H\!H}
\newcommand{\EEnd}{\operatorname{\mathbb{E}nd}}
\renewcommand{\AA}{\mathbb{A}}
\newcommand{\HH}{\mathbb{H}}
\newcommand{\PP}{\mathbf{P}}

\newcommand{\C}{\ensuremath{\mathbb{C}} }
\newcommand{\Z}{\ensuremath{\mathbb{Z}} }
\newcommand{\R}{\ensuremath{\mathbb{R}} }
\newcommand{\N}{\ensuremath{\mathbb{N}} }
\newcommand{\Su}{\Sigma}

\newcommand{\Cp}{{\ddot\C}}

\begin{document}
\def\E{\ifmmode{\mathbb E}\else{$\mathbb E$}\fi} %natural numbers
\def\N{\ifmmode{\mathbb N}\else{$\mathbb N$}\fi} %natural numbers
\def\R{\ifmmode{\mathbb R}\else{$\mathbb R$}\fi} %real numbers
\def\Q{\ifmmode{\mathbb Q}\else{$\mathbb Q$}\fi} %rational numbers
\def\C{\ifmmode{\mathbb C}\else{$\mathbb C$}\fi} %complex numbers
\def\H{\ifmmode{\mathbb H}\else{$\mathbb H$}\fi} %complex numbers
\def\Z{\ifmmode{\mathbb Z}\else{$\mathbb Z$}\fi} %integers
\def\P{\ifmmode{\mathbb P}\else{$\mathbb P$}\fi} %real numbers
\def\T{\ifmmode{\mathbb T}\else{$\mathbb T$}\fi} %real numbers
\def\SS{\ifmmode{\mathbb S}\else{$\mathbb S$}\fi} %real numbers
\def\DD{\ifmmode{\mathbb D}\else{$\mathbb D$}\fi} %real numbers

\renewcommand{\a}{\alpha}
\renewcommand{\b}{\beta}
\renewcommand{\d}{\delta}
\newcommand{\D}{\Delta}
\newcommand{\e}{\varepsilon}
\newcommand{\g}{\gamma}
\newcommand{\G}{\Gamma}
\newcommand{\la}{\lambda}
\newcommand{\La}{\Lambda}
\newcommand{\n}{\nabla}
\newcommand{\var}{\varphi}
\newcommand{\s}{\sigma}
\newcommand{\Sig}{\Sigma}
\renewcommand{\t}{\tau}
\renewcommand{\th}{\theta}
\renewcommand{\O}{\Omega}
\renewcommand{\o}{\omega}
\newcommand{\z}{\zeta}

\newcommand{\ben}{\begin{enumerate}}
\newcommand{\een}{\end{enumerate}}
\newcommand{\be}{\begin{equation}}
\newcommand{\ee}{\end{equation}}
\newcommand{\bea}{\begin{eqnarray}}
\newcommand{\eea}{\end{eqnarray}}
\newcommand{\bc}{\begin{center}}
\newcommand{\ec}{\end{center}}

\newtheorem{thm}{Theorem}[section]
\newtheorem{cor}[thm]{Corollary}
\newtheorem{lem}[thm]{Lemma}
\newtheorem{prop}[thm]{Proposition}
\newtheorem{ax}{Axiom}
\newtheorem{conj}[thm]{Conjecture}

\theoremstyle{definition}
\newtheorem{defn}{Definition}[section]

\theoremstyle{remark}
\newtheorem{rem}{\rm\bfseries{Remark}}[section]
\newtheorem*{notation}{Notation}

\newtheorem{ques}{\rm\bfseries{Question}}[section]
\newtheorem{cons}[rem]{\rm\bfseries{Construction}}
\newtheorem{exm}[rem]{\rm\bfseries{Example}}

%\numberwithin{equation}{section}

      % Includes the theorem environments
%==========================================
%% Do not edit the following command
%\setcounter{page}{1}
%\volume{13}
%==========================================
\title[Non semi-simple TQFTs]
{Non semi-simple  TQFTs from unrolled quantum $sl(2)$} 
\author[Blanchet]{Christian Blanchet}
\address{Univ Paris Diderot, Sorbonne Paris Cit\'e, IMJ-PRG, UMR 7586 CNRS, F-75013, Paris, France} 
\email{christian.blanchet@imj-prg.fr}

\author[F. Costantino]{Francesco Costantino}
\address{
Univ Toulouse III Paul Sabatier, Institut de Mathématique de Toulouse
 F-31062 TOULOUSE, France }
\email{francesco.costantino@math.univ-toulouse.fr}

\author[N. Geer]{Nathan Geer}
\address{Mathematics \& Statistics\\
  Utah State University \\
  Logan, Utah 84322, USA}
\email{nathan.geer@gmail.com}

\author[B. Patureau-Mirand]{Bertrand Patureau-Mirand}
\address{Univ. Bretagne - Sud,  UMR 6205, LMBA, F-56000 Vannes, France}
\email{bertrand.patureau@univ-ubs.fr}

\begin{abstract}
  Invariants of 3-manifolds from a non semi-simple category of modules over a
  version of quantum $\slt$ were obtained by  the last three authors
  in \cite{CGP1}. They are invariants of $3$-manifolds together with a
  cohomology class which can be interpreted as a line bundle with flat
  connection. In \cite{BCGP} we have extended those invariants to
  graded TQFTs on suitable cobordism categories.  Here we give an
  overview of constructions and results, and describe the TQFT vector
  spaces. Then we provide a new, algebraic, approach to the computation of these vector spaces.  \end{abstract}

\maketitle
% \setcounter{tocdepth}{3}
% \tableofcontents

\section*{Introduction}
New quantum invariants of $3$-manifolds equipped with $1$-dimensional
cohomology classes over $\C/2\Z$, or equivalently $\C^*$ flat
connections, have been constructed in \cite{CGP1} from a variant of
quantum $\slt$.  This family of invariants is indexed by integers
$r\geq 2$, $r\not\equiv 0$ mod $4$, which gives the order of the
quantum parameter.  In the case $r\equiv 0$ mod $4$, we have obtained
in \cite{BCGP2} invariants of $3$-manifolds equipped with generalised
spin structures corresponding to certain flat connections on the
oriented framed bundle.  These invariants are built from surgery
presentations and have common flavor with the famous
Witten-Reshetikhin-Turaev quantum invariants, but are indeed very
different. First, they are defined for $3$-manifolds equipped with
cohomology classes, and second they use a stronger version of quantum
$sl(2)$ which in particular avoids the semi-simplification procedure
required for producing modular categories.  To emphasize the power of
these new invariants we quote that for the smallest root of unity,
$r=2$, the multivariable Alexander polynomial and Reidemeister torsion
are recovered, which allows us to reproduce the classification of lens
spaces, see \cite{BCGP}.

The TQFT extension of these invariants has been carried out in
\cite{BCGP}. The main achievement is a functor on a category of {\em
  decorated cobordism} with values in finite dimensional graded vector
spaces.  An object in this category is a surface equipped with the
following data: a base point on each connected component, a possibly
empty set of colored points, a $1$-dimensional cohomology class over
$\C/2\Z$ and a Lagrangian subspace.  A morphism is a cobordism with: a
colored ribbon graph, a cohomology class and a signature defect which
all satisfy certain admissibility conditions.  A description of these
TQFT vector spaces split into two cases, depending if the cohomology
class is integral or not. In the non-integral case, it can be done
using colored trivalent graphs with a pattern similar to the
Witten-Reshetikhin-Turaev case. In the integral case we are able to prove finite
dimensionality in general and to provide a Verlinde formula for their
graded dimension under the further assumption that the surface
contains a point with a projective color.  A new result in this paper
is an Hochschild homology description of the TQFT vector spaces. In
the integral case, our statement is proved under the previous
assumption.

\section{Unrolled quantum $sl(2)$ and modified trace invariant}
In this section we recall the unrolled quantum $sl(2)$ at root of
unity.  In the whole paper, $r\geq 2$ is an integer which is non zero
modulo $4$, $r'=r$ if $r$ is odd and $r'=\frac{r}{2}$ else.

Let $q=\eexp^{\frac{i\pi}r}$, $q^x=\eexp^{\frac{ix\pi}r}$ for $x\in
\C$. Recall (see \cite{CGP3}) the $\C$-algebra $\Ubar$ given by
generators $E, F, K, K^{-1}, H$ and relations:
\begin{align*}\label{E:RelDCUqsl}
  KEK^{-1}&=q^2E, & KFK^{-1}&=q^{-2}F, &
  [E,F]&=\frac{K-K^{-1}}{q-q^{-1}}, & E^r&=0,\\
  HK&=KH,  & [H,E]&=2E, & [H,F]&=-2F,& F^r&=0.
\end{align*}
The algebra $\Ubar$ is a Hopf algebra where the coproduct, counit and
antipode are defined in \cite{CGP3}. A \emph{weight module} is a
finite dimensional module which splits as a direct sum of $H$-weight
spaces and is such that $K$ acts as $q^H$.

The category $\cat$ of  weight modules is $\C/2\Z$-graded (by
the weights modulo $2\Z$) that is \mbox{$\cat=\bigoplus_{\wb
  \alpha\in\C/2\Z}\cat_{\wb\alpha}$} and
$\otimes:\cat_{\wb\alpha}\times\cat_{\wb\beta}\to\cat_{\wb\alpha+\wb\beta}.$
 The category
$\cat$ is a ribbon category and  we have the usual
Reshetikhin-Turaev functor from $\cat$-colored ribbon graphs to $\cat$
(which is given by Penrose graphical calculus).

The simple modules in $\cat$ are highest weight modules with any
complex number as highest weight. The generic simple modules are those
which are projective. They are indexed by the set
\begin{equation*}
  \Cp=(\C\setminus\Z)\cup r\Z.  
\end{equation*}
For $\alpha\in\Cp$, the $r$-dimensional module
$V_\alpha\in\cat_{\wb{\alpha+r-1}}$  is the
irreducible module with highest weight $\alpha+r-1$.

The group of invertible modules is generated by the one dimensional
vector space denoted $\ve=\C_r$ with $H$-weight equal to $r$ and
degree equal to $r$ modulo $2$.  The subgroup of invertible objects
with trivial degree is generated by $\sigma=\C_{2r'}$, the one
dimensional vector space with $H$-weight equal to $2r'$ (if $r$ is
even then $\sigma=\ve$).  For each integer $j$, $0\leq j\leq r-1$ the
simple module with highest weight $j$ is $j+1$ dimensional. For $0\leq
j<r-1$ it is not projective, but has a $2r$-dimensional projective
cover $P_j$. The non simple indecomposable projective modules are the
$P_j\otimes \C_{r}^{\otimes k}$, $0\leq j<r-1$, $k\in \Z$.

The link invariant underlying our construction is the re-normalized
link invariant (\cite{GPT}) that we recall briefly.  The modified
dimension is the function defined on $\{V_\alpha\}_{\alpha\in\Cp}$ by
$$\qd(\alpha)=(-1)^{r-1}\frac{r\qn\alpha}{\qn {r\alpha}},$$
where $\qn\alpha= 2i\sin\frac{\pi \alpha}{r}$.  Let $L$ be a
$\cat$-colored oriented framed link in $S^3$ with at least one
component colored by an element of $\{V_\alpha:\alpha\in\Cp\}$.
Opening such a component of $L$ gives a 1-1-tangle $T$ whose open
strand is colored by some $\alpha\in\Cp$ (here and after we identify
$\Cp$ with the set of coloring modules $\{V_\alpha\}$).  The
Reshetikhin-Turaev functor associates an endomorphism of $V_\alpha$ to
this tangle.  As $V_\alpha$ is simple, this endomorphism is a scalar
$\brk T\in\C$.  The modified invariant is $F'(L)=\qd(\alpha)\brk T$.
\begin{thm}[\cite{GPT}]
  The assignment $L\mapsto F'(L)=\qd(\alpha)\brk T$ described above is an isotopy
  invariant of the colored framed oriented link $L$.
\end{thm}
The modified invariant $F'$ extends to $\cat$-colored ribbon graphs
provided the given colored graph contains at least one edge with
projective color.  In the case where the selected projective color is
simple, then the above description holds.  The general case is
obtained by using the fact that any indecomposable projective module
is a submodule in the tensor product of two simple projective ones,
see \cite{CGP3}.

\section{Invariant of closed $3$-manifolds}

We will consider here (compact and oriented) closed $3$-manifolds with
decorations and admissibility conditions. A decoration in a closed
$3$-manifold $M$ is a $\cat$-colored ribbon graph $T$ together with a
cohomology class $\coh\in H^1(M\setminus T,\C/2\Z)$ such that the
coloring of $T$ is compatible with $\coh$, i.e. each oriented edge $e$
of $T$ is colored by an object in $\cat_{\coh(m_e)}$ where $m_e$ is
the oriented meridian of $e$.
\begin{defn}\label{admissible}
  A connected decorated manifold $(M,T,\coh)$ is {admissible} if
  either there exists a loop $\gamma$ such that $\coh(\gamma) \notin
  \Z/2\Z$, or the colored ribbon graph $T$ contains at least one edge
  whose color is a projective module.
\end{defn} 
\begin{defn}
  A surgery presentation of a connected decorated $3$-manifold is a
  triple $(L,T,c)$ where
  \begin{enumerate}
  \item $L$ is an oriented framed link in $S^3$.
  \item $T$ is a $\cat$-colored ribbon graph in $S^3\setminus L$,
  \item $c$ is a cohomology class in $H^1(S^3\setminus (L\cup
    T),\C/2\Z)$ which is compatible with the coloring of $T$ and
    vanishes on the preferred parallels of components of $L$.
  \end{enumerate}
  The surgery presentation is computable if either $L$ is empty and
  $(S^3,T,c)$ is admissible, or for all meridian $m_j$ of components
  of $L$, one has $c(m_j)\notin {\Z/2\Z}$.
\end{defn}
The decorated $3$-manifold represented by $(L,T,c)$ is
$(S^3(L),T,\coh)$ where $S^3(L)$ is the manifold obtained by doing
surgery on $L$ and $\coh$ is the unique extension of $c$.  The
following proposition allows us to use computable surgery
presentations in order to define invariants of admissible decorated
$3$-manifolds. A weak version of this proposition is proved in
\cite{CGP1}.  In \cite{BCGP} the admissibility condition was slightly
extended to the case of integral cohomology classes and ribbon graphs
containing at least one projective color.  It is shown that up to local
skein equivalence (relations on ribbon graphs living in a ball) we may
suppose that the cohomology class is non integral, see \cite{BCGP}.
\begin{prop}
  Any connected admissible decorated $3$-manifold is skein equivalent
  to a decorated $3$-manifold which admits a computable surgery
  presentation. %\\
\end{prop} 
Let $\Hr=\{1-r,3-r,\ldots,r-1\}$.
For $\alpha\in \C\setminus\Z$ we define the Kirby color
$\Omega_{{\alpha}}$ as the formal linear combination
\begin{equation*}
  \label{eq:Om}
  \Omega_{{\alpha}}=\sum_{k\in \Hr}\qd(\alpha+k)V_{\alpha+k}.
\end{equation*}
Its degree is $\overline \alpha\in \C/2\Z$. Although the Kirby color depends on a complex number, its contribution in the formula for the $3$-manifold
invariant depends only on its degree.

\begin{thm}[\cite{CGP1, BCGP}]\label{theorem:CGP}
  Let $(M,T,\coh)$ be a connected admissible decorated $3$-manifold.
  If $(L,T,c)$ is a computable surgery presentation of $(M,T,\coh)$
  then
  $$\Nr(M,T,\coh)=\dfrac{F'(L\cup T)}{\Delta_+^{p}\ \Delta_-^{s}}$$
  is a well defined topological invariant (i.e. depends only of the
  orientation preserving diffeomorphism class of the triple
  $(M,T,\coh)$), where $\Delta_+,\ \Delta_-$ are given in Equation
  \eqref{eq:delta} below, $(p,s)$ is the signature of the linking
  matrix of the surgery link $L$ and for each $i$ the component $L_i$
  with meridian $m_i$ is colored by a Kirby color of degree $c(m_i)$.
\end{thm}
The TQFT construction uses a renormalized version of $\Nr$ which is
defined for possibly disconnected admissible decorated $3$-manifolds
with extended structures.  For a connected $3$-manifold, the extended
structure is an integer, interpreted as a signature defect, which
fixes the so called framing anomaly.

The renormalised invariant $\Zr$ is multiplicative for disjoint union,
and in the connected case is defined by

\begin{equation}\label{eq:Zdefi}
\Zr(M,T,\coh,n)=\eta\lambda^{b_1(M)}\dep^n \Nr(M,T,\coh)\end{equation}
where $n$ is the extended structure and $b_1(M)$ is the
first Betti number.  The
scalars $\lambda$, $\eta$ and $\dep$ are given by 
\begin{equation}
  \label{eq:delta}
\lambda=\frac{\sqrt{r'}}{r^2}, \quad\quad\eta=\frac1{r\sqrt{r'}} \quad\text{ and }\quad  \dep=\lambda\Delta_+=(\lambda\Delta_-)^{-1}=q^{-\frac32}\eexp^{-i(s+1)\pi/4}
\end{equation}
where  $ s $  is in $ \{1,2,3\} $ with $ s \equiv r $  mod $4$.
Another expression is the following:
\begin{equation}
  \label{eq:Zrconnected}
  \Zr(M,T,\omega,n)=\eta\lambda^{m}\dep^{-\sigma+n}F'(L\cup T)
\end{equation}
where the components of $L$ are colored by Kirby colors as in Theorem \ref{theorem:CGP}, $m\in\N$ is the number of components of $L$, $\sigma\in\Z$ is the
signature of the linking matrix $\lk(L)$. 

\section{The category of decorated  cobordisms}
In this section we define the category of decorated cobordisms, which
is a $2+1$-cobordism category.  The general idea is to be able to cut
admissible decorated extended $3$-manifolds along surfaces and get
gluing formulas for the invariant $\Zr$.  Note that the definition of
a decorated cobordism contains an admissibility condition.

\begin{defn}[Objects]
  A \emph{decorated surface} is a $4$-tuple $\ds{\Su}=(\Su,\{p_i\},\coh,
  {\Lag})$ where:
\begin{itemize}
\item $\Su$ is a closed, oriented surface which is 
  an ordered disjoint union of connected surfaces each having a distinguished base point $*$; 
\item $\{p_i\}$ is a finite (possibly empty) set of homogeneous
  $\cat$-colored oriented framed points in $\Su$ distinct from the
  base points, i.e. each $p_i\in \Su$ is equipped with a sign, a
  non-zero tangent vector and a color which is an object of $
  \cat_{\wb{\alpha_i}}$ for some $\wb{\alpha_i}\in \C/2\Z$;
\item $\coh\in H^1(\Su\setminus \{p_1,\ldots, p_k\}, *;\C/2\Z)$ is a
  cohomology class such that \mbox{$\coh( m_i)=\overline{\alpha_i}$}
  where $m_i$ is an oriented meridian around $p_i$;
\item ${\Lag}$ is a Lagrangian subspace of $H_1(\Su;\R)$.
\end{itemize}
\end{defn} 

The \emph{opposite or negative} of $\ds \Su=(\Su,\{p_1,\ldots
p_k\},\coh,{\Lag})$ is defined as
\mbox{$\overline{\ds\Su}=(\overline{\Su},\{\overline{p}_1,\ldots
  \overline{p}_k\},\coh,{\Lag})$} where $\overline{\Sigma}$ is
$\Sigma$ with the opposite orientation, $\overline{p}_i$ is $p_i$ with
the same vector and color, but opposite sign.

%%%% %%%%%%%%%%%%%%%%%%%%%%%%%%%%%%%%%%%%%%%%%%%%%%%%%%%%%
%\subsection*{Decorated cobordisms}\label{S:deco3}
We orient the boundaries of manifolds using the ``outward vector
first'' convention.  Here, we also use this convention for the
boundary of a graph.  Reshetikhin-Turaev ribbon functor $F$ is adapted
to this convention, so that an upward arc colored by $V$ represents
the identity of $V$.

\begin{defn}[Cobordisms]\label{D:Cobordisms}
  Let $\ds{\Su}_\pm=(\Su_\pm,\{p_i^\pm\},\coh_\pm,{\Lag}_\pm)$ be
  decorated surfaces.  A \emph{decorated cobordism} from $\ds{\Su}_-$
  to $\ds{\Su}_+$ is a $5$-tuple $\ds M=(M,T,f, \coh,n)$ where:
\begin{itemize}
\item $M$ is an oriented $3$-manifold with boundary $\partial M$;
\item $f: \overline{\Su_-}\sqcup \Su_+\to \partial M$
is a diffeomorphism that
  preserves the orientation;
 denote the image under $f$ of the base points of 
  $\overline{\Su_-}\sqcup \Su_+$ by $*$; 
\item $T$ is a $\cat$-colored ribbon graph in $M$ such that 
$\partial T=\{\wb{f(p_i^-)}\}\cup \{f(p_i^+)\}$
and the color of the edge of $T$ containing $f(p_i^\pm)$ equals the color of $p_i^\pm$; 
\item $\coh\in H^1(M\setminus T,*;\C/2\Z)$ is a cohomology class
  relative to the base points on $\partial M$, such that
  the restriction of $\coh$ to $(\partial M\setminus \partial T)\cap \Su_\pm$ is  $(f^{-1})^*(\coh_{\pm})$;
\item the coloring of $T$ is compatible with $\coh$, i.e. each oriented edge
  $e$ of $T$ is colored by an object in $\cat_{\coh(m_e)}$ where $m_e$ is the
  oriented meridian of $e$;
\item $n$ is an arbitrary integer
  called the {\em signature-defect} of $\ds M$;
\item each connected component of $M$ disjoint from 
$f(\overline{\Su_-})$ is   admissible.\footnote{Here we extend Definition \ref{admissible}  to the case $\partial M\neq \emptyset$.}
\end{itemize}
\end{defn}

 We consider decorated cobordisms from $\ds{\Su}_-$ to $\ds{\Su}_+$  up to diffeomorphism:  a \emph{diffeomorphism}  $$g:~(M,T,f,\omega,n)\to   (M',T',f',\omega',n')$$ is  an orientation preserving diffeomorphism of the underlying manifolds $M$ and $M'$, still denoted by $g$,  
  such that $g(T)=T'$, $g\circ f=f'$, $\omega=g^*(\omega')$ and 
  $n=n'$.  Remark
  that up to diffeomorphism, $\ds M=(M,T,f,\omega,n)$ only depends on $f$
  up to isotopy.

Notice that the last condition in Definition \ref{D:Cobordisms}
ensures that all connected decorated cobordisms from $\emptyset$ to 
 $\ds{\Su}_2$ are asked to be
admissible.

\begin{rem}
  Two different cohomology classes that are compatible with the same pair
  $(M,T)$ differ by an element of $H^1(M,\partial M;\C/2\Z)$.  When this group
  is zero the compatible cohomology class is unique. 
\end{rem}

Given a decorated cobordism $\ds M=(M,T,f, \coh,n)$, from $\ds{\Su}_-$
to $\ds{\Su}_+$, let $f^-=f_{|\wb{\Su_-}}:\wb{\Su_-}\to M$ and
$f^+=f_{|\Su_+}:\Su_+\to M$ be the components of $f$. If $\Lag_-$ and
$\Lag_+$ are the Lagrangians of $\ds{\Su}_-$ and $\ds{\Su}_+$,
respectively, then we define two other Lagrangians
$M_*(\Lag_-)=(f^+)_*^{-1}\bp{(f^-\circ
  \wb{\Id_{\Su_-}})_*(\Lag_-)}\subset H_1(\ds{\Su}_+;\R)$ and
$M^*(\Lag_+)=(f^-\circ
\wb{\Id_{\Su_-}})_*^{-1}\bp{(f^+)_*(\Lag_+)}\subset
H_1(\ds{\Su}_-;\R)$ where $\wb{\Id_{\Su_-}}:\Su_-\to\wb{\Su_-}$ is the
identity.  The fact that $M_*(\Lag_-)\subset H_1(\Su_+,\R)$ and
$M^*(\Lag_+)\subset H_1(\Su_-,\R)$ are indeed Lagrangians follows from
the analysis of the Lagrangian relations induced by $f_+$ and $f_-$
(see \cite[Section IV.3.4]{Tu}).

\begin{defn}\label{D:Compostion}
 For $j=1,2$,
  let $\ds M_j=(M_j,T_j,f_j, \coh_j,n_j)$ be decorated cobordism such
  that $\ds \Su=(\Su,\{p_i\},\coh_\Su,\Lag)=\partial_+ \ds
  M_1=\overline{\partial_-\ds M_2}$.  
  Let ${\Lag}_-$ and ${\Lag}_+ $  
  be the Lagrangian subspaces of the decorated surfaces
  $\wb{\partial_-\ds M_1}$ and $\partial_+ \ds M_2$, respectively.
  The composition $\ds M_2\circ \ds M_1$ is the decorated cobordism
  $(M,T,f,\coh,n)$ from $\wb{\partial_-M_1}$ to $\partial_+M_2$
  defined as follows:
\begin{itemize}
\item $M=M_1\cup_{\Su} M_2 $ is the manifold obtained by gluing
  $\partial_+ \ds M_1$ and $\wb{\partial_-\ds M_2}$ through the map
  $f_2^-\circ (f_1^+)^{-1}$,
\item $T$ is the $\cat$-colored ribbon graph $T_1\cup_{\{p_i\}} T_2$
  in $M$,
\item $f=f_1^-\sqcup f_2^+$,
\item $\coh$ is the cohomology class obtained using the Mayer-Vietoris
  map and forgetting the base points in the gluing surface,
\item $n=n_1+n_2-\mu({M_1}_*(\Lag_-),\Lag,{M_2}^*(\Lag_+))$ where $\mu$
  is the Maslov index of the Lagrangian subspaces of $H_1(\Su,\R)$.
\end{itemize}
\end{defn}

As we now explain, decorated surfaces and cobordisms form a tensor
category $\Cob$.  The objects of $\Cob$ are decorated surfaces and
morphisms are diffeomorphism classes
of decorated cobordisms.  The composition is given in Definition
\ref{D:Compostion}.  The tensor product of $\Cob$ is given by the
disjoint union: if $\ds{\Su}_j=(\Su_j,\{p_i^j\},\coh_j,{\Lag}_j)$, for
$ j=1,2$, are two decorated surfaces then define $\ds{\Su}_1\otimes
\ds \Su_2 $ as $(\Su_1\sqcup \Su_2 ,\{p_i^1\} \sqcup
\{p_i^2\},\coh_1\oplus \coh_2,{\Lag}_1\oplus {\Lag}_2)$.  Here the
ordering of the components of $\Su_1\sqcup \Su_2$ is obtained by
putting those of $\Su_1$ first then those of $\Su_2$.  In particular,
$\ds \Su_1\sqcup \ds \Su_2\neq \ds\Su_2\sqcup\ds \Su_1$ but they are
isomorphic.  Similarly, if $\ds M_j=(M_j,T_j,f_j,\coh_j,n_j), $ for $
j=1,2$, are two cobordisms then define $\ds M_1\otimes \ds M_2 $ as
$(M_1\sqcup M_2 ,T_1\sqcup T_2 ,f_1\sqcup f_2,\coh_1\oplus \coh_2,n_1+
n_2)$.  Often, for the sake of clarity we will write $\ds M_1\sqcup
\ds M_2$ instead of $\ds M_1\otimes \ds M_2$.  The composition of
cobordisms is given in Definition \ref{D:Compostion}.

\section{Graded TQFT functor}\label{sec:TQFT}
The main achievement in \cite{BCGP} is the following theorem.
\begin{thm}
  For each $r\geq 2$, $r\not\equiv 0$ mod. $4$, there exists a
  monoidal graded TQFT functor $\VV$ from the category of decorated
  cobordisms to the category of finite dimensional $\Z$-graded vector
  spaces, which extends the renormalised invariant $\Zr$.
\end{thm}
The construction uses a variant of a universal construction initiated
in \cite{BHMV}, which we will now briefly discuss.  First, given a
decorated surface $\ds{\Sigma}$, the generators of $\VV(\ds{\Sigma})$
are all decorated cobordisms from $\emptyset$ to $\ds{\Sigma}$.
The relations are defined as the kernel of the pairing with decorated
cobordisms from $\ds{\Sigma}$ to $\emptyset$. This construction
produces finite dimensional vector spaces, but fails to be monoidal.
In order to produce the full monoidal functor $\VV$, we introduce
colored spheres $\widehat{S}_k$, $k\in \Z$.  Recall that $V_0$ is the
simple projective highest module of weight $r-1$ and, for $k\in \Z$,
$\sigma^k=\C_{2kr'}$ is the one dimensional module with weight $2kr'$.
Let $\hS_k$ be the sphere with colored points $U$, where
$$U =
\begin{cases}
((V_{0},+1),(\sigma^k,+1),(V_0,-1)), & \text{if }k\neq 0 \\
((V_{0},+1),(V_0,-1)), & \text{if }k=0.
\end{cases}
$$
The cohomology class and the Lagrangian in the object $\hS_k$ are
trivial.  Using the objects $\hS_k$ we now give a description of the
graded vector space $\VV(\ds{\Su})=\oplus_k \VV_k(\ds{\Su})$
associated with a decorated surface $\ds{\Su}$ equivalent to the
definition given in \cite{BHMV}.  The vector space $\VV_k(\ds{\Su})$
is generated by all decorated cobordisms from $\hS_k$ to $\ds{\Su}$.
Relations are defined as the kernel of the pairing with decorated
cobordisms from $\ds{\Sigma}$ to $\hS_k$.

The monoidality property uses natural cobordisms which we call
``pants'' $\Pa_{k,l}^{k+l}$ from $\hS_k\sqcup \hS_l$ to $\hS_{k+l}$
and similar morphisms ``upside down'' $\Pa_{k+l}^{k,l}$ \cite[Section
4]{BCGP}. The category of decorated cobordism has a symmetry extending
the standard flip of two connected components. We have shown in
\cite[Section 5]{BHMV} that the monoidal functor $\VV$ is compatible
with standard symmetry on vector spaces when $r$ is odd, and
supersymmetry on graded vector spaces when $r$ is even.

In genus zero the TQFT vector spaces can be described from the category $\cat$. Indeed the degree zero part for a sphere with one positive  point colored with a projective object is the space of invariants \cite[Section 6.1]{BCGP}:
\begin{equation}
%B \VV_0(S^2,(+,V))=\Hom(\C,V).
\VV_0(S^2,(V,+1))=\Hom(\C,V).
\end{equation}
If there are several points with $\pm$ sign and at least one projective color, we reduce to this case
via duality and tensor product.  It would be convenient to have such
formula for the whole graded vector space $\VV(S^2,(+,V))$. This is
done by introducing graded morphisms.  The category ${\gcat}$ has the
same objects as in $\cat$, but morphisms are $\Z$-graded vector
spaces, graded by $k$, defined by
\begin{equation}
\HHom(V,W)\cong\oplus_k \Hom(V\otimes \sigma^k,W).
\end{equation}
The composition of $f\in \Hom(U\otimes \sigma^k,V)$ with $g\in
\Hom(V\otimes \sigma^l,W)$ is $g\circ (f\otimes \Id_{\sigma^l})\in
\Hom(U\otimes \sigma^{l+k},W)$.  We then have:
\begin{equation}
\VV(S^2,(V,+1))=\HHom(\C,V).
\end{equation}
In the next section, a general description of the TQFT vector spaces will be obtained from
elementary bricks consisting of spheres with two or three points
colored with indecomposable projective modules.

For simple projective
modules, the list of bricks is the following \cite[Theorem 5.2]{CGP3}
\begin{equation}
  \VV(S^2,(V_\alpha,-),(V_\beta,+))=\HHom(V_\alpha,V_\beta)=
  \left\{
    \begin{array}{l}\ensuremath{
      \C\text{ if } \beta-\alpha=2kr'\text{ with }k\in \Z \text{ giving degree},}\\
      0 \text{ else}.
    \end{array}
  \right.
\end{equation}
Recall that $H_r=\{-r+1,-r+3,\dots,r-1\}$ and that $\ve,\sigma^k$ are the invertible 
modules of weight $r$ and $2kr'$ respectively.
\begin{equation}
\VV_k(S^2,(V_\alpha,+),(V_\beta,+),(V_\gamma,+))=\Hom(\sigma^k,V_\alpha\otimes V_\beta\otimes V_\gamma)=\left\{\begin{array}{l}\ensuremath{\C\text{ if } \alpha+\beta+\gamma\in H_r+ 2r'k,}\\
0 \text{ else}.\end{array}\right.
\end{equation}
In the first case above we will say that the triple
$(\alpha,\beta,\gamma)$ is $r$-admissible for degree $k$.  In other words, a
triple $(\alpha,\beta,\gamma)$ is $r$-admissible (for some degree
$k\in\Z$) if and only if $\alpha+\beta+\gamma+r-1\equiv 0\text{ mod
}2\Z$; remark that such $k$ is uniquely determined if $r$ is odd and
that there are two possible values of $k$ (differing by $1$) if $r$ is
even.  So, given a trivalent graph $G$ with oriented edges, an
$r$-admissible coloring is a map $c:{\rm Edges}(G)\to \C$ such that
the colors incoming to each vertex (if an edge is outgoing we consider it
as an incoming edge with opposite color) form an $r$-admissible
triple and if $r$ is even the datum of a specification of the
degree of the triple around each vertex. The total degree of an
$r$-admissible coloring on $G$ is the sum of the degrees of the
triples at the vertices.  
 We get a basis for the TQFT vector space using $r$-admissible colorings in
generic case as follows :

\begin{thm}\label{generic_surface}
  Let $G$ be an oriented uni-trivalent planar graph and $\Sigma=\partial
  H_G$ be the boundary of a regular neighborhood of $G$.  Suppose
  $\Sigma$ is equipped with a cohomology class which is non integral
  on all meridians of edges of $G$.  Then a basis of $\VV(\Sigma)$ is
  indexed by $r$-admissible colorings of $G$, where colors are
  compatible with the evaluation on meridians and stay in the range
  $]-r,r]$ if $r$ is odd and $[0,r[$ if $r$ is even.  Moreover, the
  degree is the sum of degrees at the vertices.
\end{thm}
The theorem above is an improved version of \cite[Section 6.3]{BCGP} and can be proved using surgery formulas and skein
methods.  It can also be deduced from the general splitting
theorem \ref{trace} in the next section. 
%In fact our TQFT extends to a $2$-functor; this will be discussed elsewhere.
\begin{rem}
  Here for the skein description of the basis we need to add a germ of
  incoming edge at each trivalent vertex corresponding to the incoming
  $\widehat S_k$. For $r$ even, due to supersymmetry property, we also
  need to fix an ordering of trivalent vertices up to even
  permutation. We give an example in Figure \ref{coloring}. Here an
  admissible coloring is given by the complex values $\alpha_i$ in the
  fundamental band and integers $k_j$ satisfying the above $r$-admissibility conditions.
\end{rem}
\begin{figure}
\includegraphics[height=5cm]{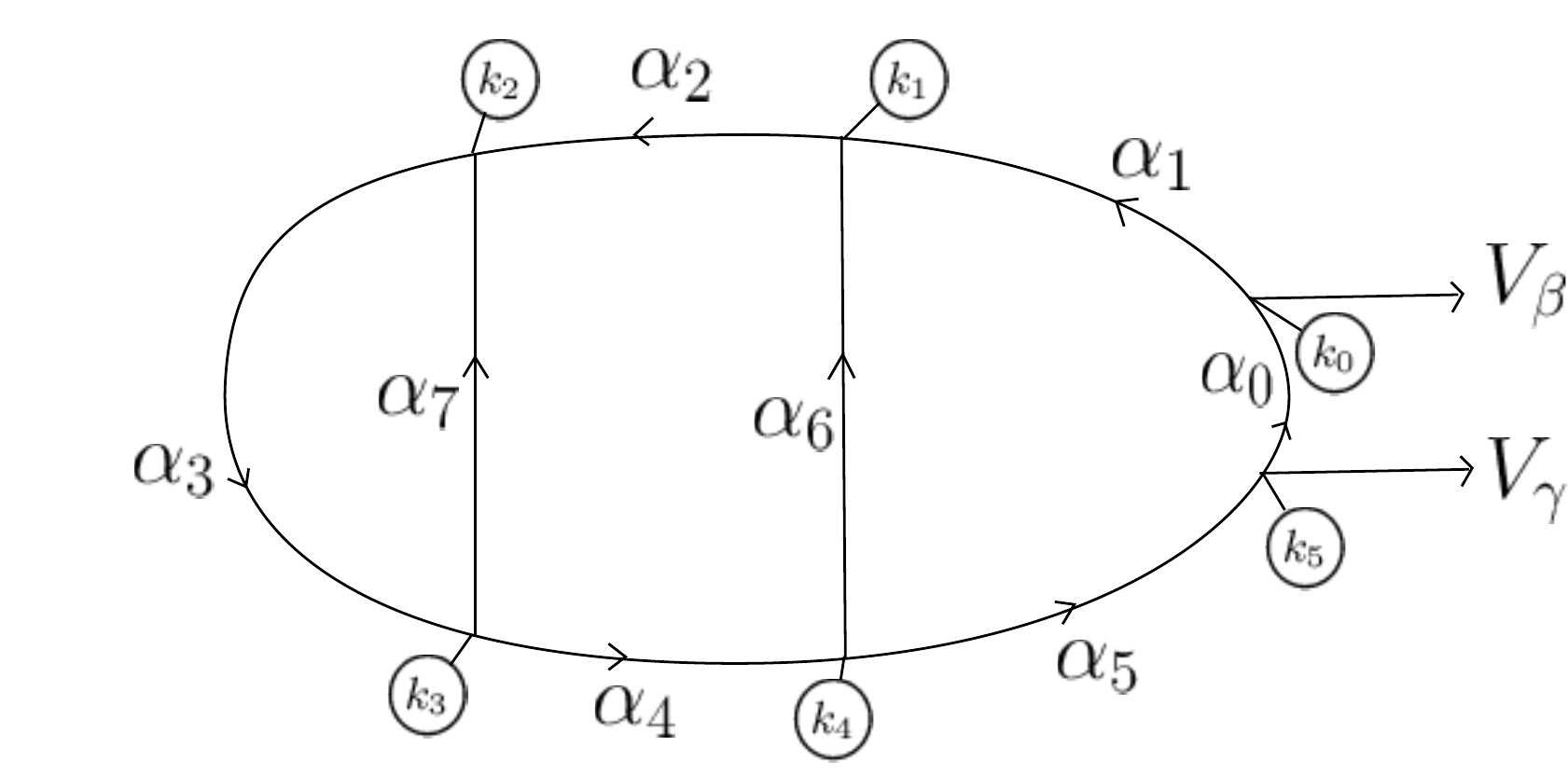} 
\caption{Colored graph in generic case}
\label{coloring}
\end{figure}
\begin{cor}
  The dimension for the TQFT vector space of a genus $g$ surface with
  non integral cohomology class is $r^{3g-3}$ if $r$ is odd, and
  $\frac{r^{3g-3}}{2^{g-1}}$ if $r$ is even.
\end{cor}

\subsection*{Verlinde formula and the graded dimensions of $\VV(\Su)$} 
\begin{defn}\label{admissible_surface}
A decorated surface \mbox{$\ds\Su=(\Su,\{p_i\},\coh,\Lag)$} is called
admissible if either there exists a loop $\gamma$ such that
$\coh(\gamma) \notin {\Z/2\Z}$, or at least one of the $p_i$ has
projective color. 
\end{defn}

 In this section, we give the Verlinde formula
giving the graded dimension of the TQFT vector spaces for admissible
surfaces.

The  proof is given in \cite[Theorem 5.9]{BCGP} and uses the pivotal structure which does not fully
exists on the decorated cobordism category because the standard
coevaluation cobordism is admissible only for { admissible surfaces}.
 We do not know the Verlinde formula for non admissible surfaces.

Let $\ds\Su=(\Su,\{p_i\},\coh,\Lag)$ be a connected admissible
decorated surface with base point $*$.  For each ${\wb\beta}\in
\C/2\Z$ let $\Id_\Su^\vp:\Su\to\Su$ be the product $[0,1]\times
(\Su,\{p_i\})$ 
with cohomology class $\vp$ which restricts to $\coh$
%B whith cohomology class $\vp$ which restricts to $\coh$
on each boundary component and whose evaluation on the relative cycle
$[0,1]\times *$ is equal to $\wb\beta$. The mapping torus of
$\Id_\Su^\vp$, obtained by gluing the two ends is denoted by
$S^1_{\wb\beta}\times \ds\Su$. 
\begin{thm}[Verlinde formula for graded dimensions] \label{P:verlinde}
  Let $\ds\Su=(\Su,\{p_1,\ldots p_n\},\coh,{\Lag})$ be a connected admissible
  surface of genus $g$.
  Define the graded (super)-dimension of $\VV(\ds\Su)$ by
  $$\dim_t(\VV(\ds\Su))=\left\{
    \begin{array}{ll}
      \sum_{{\de}\in \Z} \dim(\VV_{{\de}}(\Su))t^{\de} &\text{ if $r$ odd},\\
      \sum_{{\de}\in \Z} (-1)^{\de}\dim(\VV_{{\de}}(\Su))t^{\de} &\text{ if $r$ even.}
    \end{array}\right.$$
  $$\text{Then}\qquad\dim_{q^{2r'\beta}}(\VV(\Su))=\Zr( S^1_{\wb\beta}\times\ds\Su).$$
 If $p$ is empty, then 
  $$\Zr(S^1_{\wb\beta}\times \ds\Su)=\frac1r(r')^{g} \sum_{k\in
    \Hr}\bp{\frac{\qn{r\beta}}{\qn{\beta+k}}}^{2g-2}$$
    where $\{\gamma\}=q^\gamma-q^{-\gamma}$.\\
  If $p_i$ is colored by $V_{c_i}$ for $c_1,\ldots,c_n\in\Cp$ then 
  \begin{displaymath}
    \Zr( S^1_{\wb\beta}\times \ds\Su)=\frac{(-1)^{n(r-1)}}r(r')^{g}q^{c\beta}\sum_{k\in \Hr}q^{ck}  
    \bp{\frac{\qn{r\beta}}{\qn{\beta+k}}}^{2g-2+n}
  \end{displaymath}
  where $c=\sum_ic_i$.        
\end{thm}

\section{Description of graded TQFT vector spaces for admissible surfaces}
This section contains new results. Our aim is to describe the graded TQFT vector spaces for all admissible surfaces (Definition \ref{admissible_surface}). This includes the case of a surface with vanishing cohomology class and one point colored with projective object of degree $0$, e.g $P_0$ or $V_0$ in case of odd $r$.

\subsection*{Projective objects and multiplicity modules}
For each grading $\alpha\in \C/2\Z$ we define a projective object
$\PP_\alpha$ which is a direct sum of $r'$ indecomposable projective
ones representing orbits under action of degree $\overline{0}\in
\C/2\Z$ invertible objects.

For $r$ odd and $\alpha\in \C/2\Z$ we define $\PP_\alpha$  as follows: 
$$
\PP_{\alpha}:=\left\{\begin{tabular}{cc}
    $P_0\oplus P_2\oplus\dots \oplus P_{r-3} \oplus V_0\oplus\left( \ve\otimes (P_1\oplus P_3\oplus \dots \oplus P_{r-2})\right)$ & {\rm if} $\alpha=\overline{0}\in \C/2\Z$\\
    $P_1\oplus P_3\oplus \dots \oplus P_{r-2}\oplus \left(\ve\otimes (P_0\oplus P_2\oplus\dots \oplus P_{r-3} \oplus V_0)\right)$	& {\rm if} $\alpha=\overline{1}\in \C/2\Z$\\
    $\bigoplus_{\beta\equiv \alpha, \Re(\beta)\in ]-r,r]} V_\beta$ &
    {\rm else.}\end{tabular}\right.$$
	
	Here $\Re(\beta)$ denotes the real part of $\beta$.
	
For $r$ even and $\alpha\in \C/2\Z$ we define $\PP_\alpha$  as follows:
$$\PP_{\alpha}:=\left\{\begin{tabular}{cc}
    $P_0\oplus P_2\oplus\dots \oplus P_{r-2}$ & {\rm if} $\alpha=\overline{0}\in \C/2\Z$\\
    $P_1\oplus P_3\oplus \dots \oplus P_{r-3}\oplus V_0$	& {\rm if} $\alpha=\overline{1}\in \C/2\Z$\\
    $\bigoplus_{\beta\equiv \alpha, \Re(\beta)\in [0,r[} V_\beta$ & {\rm else.}\end{tabular}\right.$$ 

In all cases we set $\AA_\alpha=\EEnd(\PP_\alpha)$, and call it the degree $\alpha$ basic algebra.
 For generic $\alpha$ it is semisimple. For degree $0$ and $1$, it is described in \cite{CGP3}; we quote that it contains  nilpotent elements of order $2$ as well as non diagonal morphisms.
 
 For $\alpha$, $\beta$, $\gamma$ in $\C/2\Z$, we define the
 multiplicity module
 \mbox{$\HH(\alpha,\beta,\gamma)=\HHom(\C,\PP_\alpha\otimes\PP_\beta\otimes\PP_\gamma)$}
 which support a right action of $\AA_\alpha\otimes \AA_\beta\otimes
 \AA_\gamma$.  
 %Using duality, we get isomorphisms
%%N4: \mbox{$\HH(\alpha,\beta,\gamma)\cong \HHom(\PP_{\alpha}^*,\PP_\beta\otimes\PP_\gamma)$} ...
%$$\HH(\alpha,\beta,\gamma)\cong \HHom(\PP_{\alpha}^*,\PP_\beta\otimes\PP_\gamma)\cong
%\HHom(\C,\PP_\beta\otimes\PP_\gamma\otimes \PP_{\alpha})\cong
% \HHom(\PP_\beta^*\otimes \PP_{\alpha}^*,\PP_\gamma)...$$
 
 Remark that $ \HH(\alpha,\beta,\gamma)$ is isomorphic to
 $\VV\left(S^2,(\PP_\alpha,+),(\PP_\beta,+),(\PP_\gamma,+)\right)$.
 If we represent the pair of pants by a planar trivalent vertex, then
 we need to add a germ of incoming edge corresponding to the incoming
 $\widehat{S}_k$ in order to fix the isomorphism, and use a braiding
 with this colored edge to write the isomorphism corresponding to a
 cyclic permutation.
 
 Given a trivalent vertex $v$ with  ordered oriented edges
 labelled with degrees in $\C/2\Z$, then we have an associated
 multiplicity module $\HH(v)$ defined using $\PP$ color if the
 orientation is outgoing and the dual if the orientation is ingoing.
 For an outgoing edge, we have a right action of the algebra
 corresponding to the grading, and for ingoing edge we have left
 action.  If some edge has fixed color, then the multiplicity module
 is adapted using this color.
 
 \subsection*{TQFT vector space as degree 0 Hochschild homology}
 Now we are able to describe the TQFT vector spaces in the admissible
 case as follows.  Suppose that $G=(\mathcal{V},\mathcal{E})$ is an
 oriented uni-trivalent planar graph and $\Sigma=\partial H_G$ is the
 boundary of a tubular neighborhood of $G$.  Let
 $(\alpha_e)_{e\in\mathcal{E}}\in(\C/2\Z)^{\mathcal{E}}$ be a 1-cycle
 coloring the edges of $G$ and assume that each external edge $e$
 (adjacent to a univalent vertex) is equipped with an object of
 $\cat_{\alpha_e}$.  We identify the univalent vertices of $G$ with
 marked points on $\Sigma$.  Let $\ds \Sigma=(\Su,\{p_i\},\coh,
 {\Lag})$ be the admissible decorated surface such that $\Lag$ is
 defined by $H_G$ and the cohomology class $\coh$, whose value on the
 meridian of an edge $e\in \mathcal{E}$ is $\alpha_e$, vanishes on all
 longitudes
 %B  oriented trivalent planar graph and $\Sigma=\partial H_G$, the
 % boundary of a regular neighborhood of $G$, is equipped with coloring
 % of $1$-valent vertices Lagrangian subspace defined by $H_G$ and a
 % cohomology class $\coh$, whose value on the meridian of an edge $e\in
 % \mathcal{E}$ is $\alpha_e$, and which vanishes on all longitudes
 (curves on the surface which are the preferred parallels of the
 cycles of the graph $G$ using the blackboard framing).  
 Suppose that $\ds \Sigma$ 
% Denote by
%  $\ds \Sigma$ the corresponding decorated surface and suppose it 
 is admissible (so either one of the colors $\alpha_e$ is non integral
 or one of the colors of the $1$-valent vertices of $G$ is
 projective). 
 If $\mathbf{A}$ is an algebra and $\mathbf{M}$ is an
$\mathbf{A}$-bimodule, then the degree $0$ Hochschild homology module
$\Hoch_0(\mathbf{A},\mathbf{M})$ is defined as the quotient of the
module $\mathbf{M}$ by the relations $uv-vu$ for all 
%B: $u,v\in \mathbf{A}$.
$(u,v)\in \mathbf{A}\times\mathbf{M}$.

  Then the following holds:
\begin{thm}\label{trace}
  Let $\ds \Sigma$ be an admissible decorated surface and $G$ be a
  graph for $\ds \Sigma$ as above.\\
  a) There exists a surjective graded homomorphism $$\HH_G:=\otimes_{v\in \mathcal{V}} \HH(v)\rightarrow \VV(\ds\Sigma).$$\\
  b) Let $\mathcal{E}_\mathrm{int}$ denote the set of internal edges,
  then
%N2: this tensor product of multiplicity modules
  the tensor product $\HH_G=\otimes_{v\in \mathcal{V}} \HH(v)$
 is a bimodule over $\AA_G=\otimes_{e\in \mathcal{E}_\mathrm{int}}\AA_{\alpha_e}$ and we get an isomorphism between degree zero Hochschild homology and the TQFT vector space:
$$\Hoch_0(\AA_G,\HH_G)\cong \VV(\ds \Sigma).$$
\end{thm}
\begin{rem}
Any admissible decorated surface can be represented this way. The vanishing hypothesis on longitudes can be removed if we use twisted action of the basic algebras.
\end{rem}

This theorem is the new part of this paper.
% We give the proof in
%Section \ref{proof} below, 
We expect that a similar statement is true
for the non-admissible case, but we leave this for further
investigation.

\subsection*{Proof of Theorem \ref{trace}}\label{proof}

In order to prove the theorem, we extend our TQFT to decorated
surfaces with boundary, and follow a method from \cite{BHMV}. This is
part of an extended TQFT functor which will be fully developped in
\cite{dR}.  Let us first recall some definitions from Appendix A of
\cite{BHMV}:
\begin{defn}
\begin{itemize}
\item An algebroid is a $\C$-linear category $\Delta$; if $a,b\in
  Ob(\Delta)$ one denotes $_a\Delta_b=\Hom_{\Delta}(a,b)$ and the
  composition of $y\in {_b\Delta_c}$ and $x\in {_a\Delta_b}$ is also
  denoted $y\circ x=xy$.
\item A left $\Delta$-module is a functor $F:\Delta\to {\rm \bf
    Vect}$; a right module is a functor $\Delta^{op}\to {\rm \bf
    Vect}$. If $M$ is a left (resp. right) $\Delta$-module and $a\in
  Ob(\Delta)$, we denote $M(a)\in {\rm \bf Vect}$ also by $_aM$
  (resp. $M_a$).
\item If $\Delta'$ is another algebroid, a $(\Delta\times
  \Delta')$-bimodule is a functor $F:\Delta\times (\Delta')^{op}\to
  {\rm \bf Vect}$.
\item If $\Delta$ is an algebroid, $M$ a right $\Delta$-module and $N$
  a left $\Delta$-module, then their \emph{tensor product}\
  $M\otimes_\Delta N$ is the quotient of the vector
  space $$\bigoplus_a M_a \otimes_{\C} {}_aN$$ (where $a$ ranges over
  all the objects of $\Delta$ or if $\Delta$ is not small in a small
  skeleton of $\Delta$) by the subvector space generated by the
  relations $u\alpha\otimes v-u\otimes \alpha v$ where $u\in M_a, v\in
  _bN, \alpha\in _a\Delta_b$.
\item If $\Delta$ is an algebroid and $M$ is a $\Delta$-bimodule, the Hochschild homology module $\Hoch_0(\Delta,M)$ is defined as the quotient of the module $\bigoplus_a {}_aM_a$ by the relations $uv-vu$ for all $u\in {}_a\Delta_b,v\in {}_b\Delta_a$. 
\end{itemize}
\end{defn}
\begin{exm}
If $\Delta$ is a $k$-linear category, then $\Delta$ is a left and right $\Delta$-module. The left module structure is given by the functor $a\to \Hom(\cdot ,a)$; the right module by $a\to \Hom(a,\cdot)$.
\end{exm}

Let $\gamma$ be an oriented closed (possibly non connected) curve with
one base point $*_i$ per component and $\coh_\gamma\in
H^1(\gamma,\cup_{\pi_0 \gamma} *_i;\C/2\Z)$, then we define the
algebroid $\Delta(\gamma,\coh_\gamma)$ where objects are admissible
decorated surfaces $\ds \Sigma=(\Sigma,p,\coh,{\Lag})$ with boundary
$(\gamma,\coh_\gamma)$. Here we ask that $\cup *_i$ is the set of base
points in $\Sigma$ and the restriction of $\coh$ to $\partial \Sigma$
is $\coh_\gamma$; furthermore the Lagrangian subspace ${\Lag}$ of a
surface with boundary $\Sigma$ is by definition a maximal isotropic
subspace of $H_1(\Sigma;\R)$.  The space of morphisms from $\ds\Sigma$
to $\ds\Sigma'$ is the graded TQFT vector space $\VV(\overline{\ds
  \Sigma}\cup_\gamma{\ds \Sigma'})$.  (Here the Lagrangian subspace of
the glueing of two surfaces $\overline{\ds \Sigma}\cup_\gamma{\ds
  \Sigma'}$ is the image in $H_1(\overline{\ds \Sigma}\cup_\gamma{\ds
  \Sigma'};\R)$ of ${\Lag}_1\oplus {\Lag}_2$ via Mayer-Vietoris map.)
%\nota{Check that this really gives a Lagrangian}
Composition is induced by gluing along the intermediate decorated
surface, using also the pair of pants map $\Pa_{k,l}^{k+l}$ from
$\hS_k\sqcup \hS_l$ to $\hS_{k+l}$.  For a (possibly non admissible)
decorated surface $\ds\Sigma$ with boundary $(\gamma,\coh_\gamma)$,
which we consider as a cobordism from $\emptyset$ to
$(\gamma,\coh_\gamma)$, we define the TQFT module $\VV(\ds\Sigma)_\_$
which is a right module over $\Delta(\gamma,\coh_\gamma)$ by
$\VV(\ds\Sigma)_{\Sigma'}=\VV({\ds \Sigma}\cup_\gamma\overline{\ds
  \Sigma'})$, and similarly if $\ds \Sigma''$ is a decorated cobordism
from $(\gamma,\coh_\gamma)$ to $\emptyset$ we define the TQFT module
${}_\_\VV(\ds\Sigma)$ which is a left module over
$\Delta(\gamma,\coh_\gamma)$ by ${}_{\Sigma'}\VV(\ds\Sigma'')=\VV({\ds
  \Sigma'}\cup_\gamma{{\ds
    \Sigma''}})$.%\nota{F.C.: added an $\overline{\ds \Sigma}$ here.}
%\nota{CB: We removed the overline since the boundary of the cobordism is $\overline{\gamma}$.I changed again some orientations, please double check.}

We have the following general splitting theorem.
\begin{thm}[Splitting theorem]\label{teo:splitting}
  If $\ds \Sigma=\ds \Sigma_1\cup_\gamma \ds\Sigma_2$ is a closed
  decorated surface which splits along the multicurve $\gamma$ as a
  cobordism $\ds \Sigma_1$ from $\emptyset$ to $(\gamma,\coh_\gamma)$
  followed by a cobordism $\ds \Sigma_2$ from $(\gamma,\coh_\gamma)$
  to $\emptyset$, and suppose that $\ds \Sigma_1$ is admissible then
  we have an isomorphism
$$\VV((\ds\Sigma_1))_\_\otimes_{\Delta(\gamma,\coh_\gamma)}\  {}_\_\VV((\ds\Sigma_2))\cong \VV(\ds\Sigma)\ .$$
\end{thm}
\begin{proof}
  The map from left to right is induced by gluing and the pair of
  pants $\Pa_{k+l}^{k,l}$ (recall indeed that in Section \ref{sec:TQFT} we defined $\VV(\Su)=\bigoplus_k \VV_k(\Su)$ where $\VV_k(\Su)$ is generated by of the morphisms from $\widehat{S}_k$ to $\Su$
  %\nota{F.C. added this sentence})
  .  We have to build the inverse map.  Since
  $\ds \Sigma_1$ is admissible then the inverse map is given by the
  composition
%$$[\ds M]\mapsto \Id'_{\ds \Sigma_1}\otimes [\ds M]\in \VV(\ds\Sigma_1)_{\ds \Sigma_1}\otimes_\C\,{}_{\ds \Sigma_1}\VV(\ds\Sigma_2)\mapsto  \Id'_{\ds \Sigma_1}\otimes [\ds M]\in \VV(\ds\Sigma_1)_\_\otimes_{\Delta(\gamma,\coh_\gamma)}\  {}_\_\VV(\ds\Sigma_2)$$
$$
\begin{array}{ccccc}
 \VV(\ds\Sigma)&\to&
 \VV(\ds\Sigma_1)_{\ds \Sigma_1}\otimes_\C\,{}_{\ds \Sigma_1}\VV(\ds\Sigma_2)&\to&
 \VV(\ds\Sigma_1)_\_\otimes_{\Delta(\gamma,\coh_\gamma)}\  {}_\_\VV(\ds\Sigma_2){}
 \\
 \left[\ds M\right]&\mapsto& 
 \Id'_{\ds \Sigma_1}\otimes [\ds M]&\mapsto&  
 \Id'_{\ds \Sigma_1}\otimes [\ds M]
\end{array}
$$
% [\ds M]\mapsto \Id'_{\ds \Sigma_1}\otimes [\ds M]\in \VV(\ds\Sigma_1)_{\ds \Sigma_1}\otimes_\C\,{}_{\ds \Sigma_1}\VV(\ds\Sigma_2)\mapsto  \Id'_{\ds \Sigma_1}\otimes [\ds M]\in \VV(\ds\Sigma_1)_\_\otimes_{\Delta(\gamma,\coh_\gamma)}\  {}_\_\VV(\ds\Sigma_2)
where $\Id'_{\ds \Sigma_1}$ denote the class of the cylinder of ${\ds
  \Sigma_1}$ pinched at $\gamma\times[0,1]$. The first map consists in viewing $M$ as the glueing of a collar over ${\ds \Sigma_1}$ in $M$ and its complement which is diffeomorphic to the whole $M$; the second map is the quotient defining the tensor product over $\Delta(\gamma,\coh_\gamma)$. 
  %\nota{F.C. added this sentence}
  
This map is well defined as by definition it holds $\,{}_{\ds
  \Sigma_1}\VV(\ds\Sigma_2)=\VV(\ds\Sigma_1\cup_{\gamma} {\ds
  \Sigma_2})=\VV(\ds \Sigma)$.
\end{proof}
As a corollary we get the TQFT vector space of a decorated surface
$\ds \Sigma$ which is obtained by closing an admissible decorated
cobordism $\ds \Sigma_\gamma$ from $(\gamma,\coh_\gamma)$ to itself by identification of the two copies of $\gamma$.
Here, we have a base point on each component of the boundary curve. We
then define the TQFT bimodule $\VV({}_\_(\ds \Sigma_\gamma)_\_)$ over
$\Delta(\gamma,\coh_\gamma)$.
\begin{cor}[Trace corollary]\label{cor:trace}
  Let $\ds \Sigma$ be a connected closed surface obtained by closing
  an admissible cobordism (identification of the two connected copies
  of $\gamma$), then we have an isomorphism:
  $$\Hoch_0(\Delta(\gamma, \coh_\gamma),\VV({}_\_(\ds \Sigma_\gamma)_\_)\cong \VV(\ds \Sigma)\ .$$
\end{cor}
The above theorem and corollary cannot be used directly for the
computation. Following again the method in \cite{BHMV} we will have a
colored version of the splitting theorem and trace corollary thanks to
a Morita reduction of the algebroid of a curve.  Using the theorem in
\cite[Appendix A]{BHMV} we establish the following:
\begin{thm}
  The algebroid of a curve $\gamma=\sqcup (\gamma_i,\omega_\gamma)$
  is Morita equivalent to a tensor product of basic algebras
  $\otimes_i\AA_{\alpha_i}$, where
  $\alpha_i=\omega_\gamma(\gamma_i)$. Here, we view this algebra as an
  algebroid with $\otimes_i\PP_{\alpha_i}$ as unique object.
\end{thm}
(The above theorem is basically a special case of  
Morita's theorem showing that the category ${\rm Mod-A}$ of representations 
of a finite dimensional algebra $A$ is Morita equivalent 
to that of representations of $End(P)$ if $P$ is a projective module over $A$
generating ${\rm Mod-A}$.)
%\nota{F.C. Added this}
We deduce
 the colored splitting and colored trace corollaries. We
state them for connected curve which is enough for proving Theorem
\ref{trace}; they also hold for multicurves.

\begin{cor}[Colored splitting]\label{colored:splitting}
  If $\ds \Sigma=\ds \Sigma_1\cup_\gamma \ds\Sigma_2$ is a closed
  decorated surface which splits along the connected curve $\gamma$ as
  a cobordism $\ds \Sigma_1$ from $\emptyset$ to
  $(\gamma,\coh_\gamma)$ followed by a cobordism $\ds \Sigma_2$ from
  $(\gamma,\coh_\gamma)$ to $\emptyset$, and suppose that $\ds
  \Sigma_1$ is admissible then we have an isomorphism
  $$\VV(\ds\Sigma_1\cup (\disc,(\PP_\alpha,+)))\ \otimes_{\AA_\alpha}\  \VV(((\disc,(\PP_\alpha,-))\cup\ds\Sigma_2 )\cong \VV(\ds\Sigma)\ .$$
\end{cor}
\begin{cor}[Colored trace]\label{colored:trace}
  Let $\ds \Sigma$ be a connected closed surface obtained by closing
  along a connected curve $(\gamma,\coh_\gamma)$ an admissible
  cobordism $\ds \Sigma_\gamma$ from $(\gamma,\coh_\gamma)$ to itself,
  then we have an isomorphism:
$$\Hoch_0(\AA_\alpha,\VV((\disc,(\PP_\alpha,-)) \cup\ds \Sigma_\gamma\cup(\disc,(\PP_\alpha,+)) )\cong \VV(\ds \Sigma)\ .$$
\end{cor}
In both statements $\alpha=\omega_\gamma(\gamma)$. Note that in the
last statement we need base point on each boundary component, and have
to be careful with the relative homology class when computing left and
right action.

Theorem \ref{trace} is now proved from sphere with 2 or 3 points by
using the above results.

\end{document}